\documentclass[11pt,english]{article}
\usepackage{amsfonts,amsmath,amsthm,amscd,amssymb,latexsym,cite,verbatim,texdraw,floatflt,
caption2,pb-diagram}

\theoremstyle{plain}
\newtheorem{theorem}{Theorem}[section]
\newtheorem{lemma}{Lemma}[section]
\newtheorem{proposition}{Proposition}[section]
\newtheorem{corollary}{Corollary}[section]
\newtheorem{definition}{Definition}[section]
\theoremstyle{definition}
\newtheorem{example}{Example}[section]

\newcommand{\keywords}{\textbf{Key words. }\medskip}
\newcommand{\subjclass}{\textbf{MSC 2020. }\medskip}
\renewcommand{\abstract}{\textbf{Abstract. }\medskip}

\numberwithin{equation}{section}

\DeclareMathOperator{\diam}{diam}
\newcommand{\RR}{\mathbb{R}}

\begin{document}

\title{Quasisymmetric mappings in b-metric spaces}

\author{Evgeniy A. Petrov, Ruslan R. Salimov}



\date{}

\maketitle

\begin{abstract}
Considering quasisymmetric mappings between b-metric spaces we have found a new estimation for the ratio of diameters of two subsets which are images of two bounded subsets. This result generalizes the well-known Tukia-V\"{a}is\"{a}l\"{a} inequality. The condition under which the image of a b-metric space under quasisymmetry is also a b-metric space is established. Moreover, the latter question is investigated for additive metric spaces.
\end{abstract}

\subjclass{54E40, 54E25}

\keywords{b-metric space, quasisymmetric mapping, additive metric space}

\section{Introduction}

Quasisymmetric mappings on the real line were first introduced by Beurling and Ahlfors~\cite{BA56}.
They found a way to obtain a quasiconformal extension of a quasisymmetric self-mapping of the real axis to a self-mapping of the upper half-plane.
This concept was later generalized by Tukia and V\"{a}is\"{a}l\"{a} in their seminal paper~\cite{TV80}, who studied quasisymmetric mappings between general metric spaces.
In the recent years, these mappings are being intensively studied by many mathematicians, see, e.g.,~\cite{AKT05, T98, HL15, BM13, YVZ18, AB19}.

Recall that a \textit{metric} on a set $X$ is a function $d\colon X\times X\rightarrow \mathbb{R}^+$, $\mathbb R^+ = [0,\infty)$, such that for all $x, y, z \in X$:
\begin{itemize}
\item[(i)] $d(x,y)=d(y,x)$,
\item[(ii)] $(d(x,y)=0)\Leftrightarrow (x=y)$,
\item[(iii)] $d(x,y)\leq d(x,z)+d(z,y)$.
\end{itemize}
If inequality \((iii)\) is replaced by $d(x,y)\leq K(d(x,z)+d(z,y))$, $K\geqslant 1$, then  $(X,d)$ is called a \emph{b-metric space}. In the case if only condition (i) and (ii) hold $(X,d)$ is called a \emph{semimetric space} (see, for example, \cite[p.~7]{Bl}). Note that every metric space is a b-metric space (with any $K\geqslant 1$) and every b-metric space is a semimetric space.

Initially, the definition of b-metric space was introduced by Czerwik~\cite{C93} in 1993 as above only with the fixed $K=2$. After that, in 1998, Czerwik~\cite{C98} generalized this notion where the constant 2 was replaced by a constant $K\geqslant 1$, also with the same name b-metric. In 2010, Khamsi and Hussain~\cite{KH10} reintroduced the notion of a b-metric under the name metric-type. Many fixed point theorems in the literature were considered on b-metric spaces, see, e.g.~\cite{KH10, ABKM12, 2ABKM12, HSA14, CP14} and~\cite[pp. 113--131]{KS14}. The notion of a b-metric  admits many generalizations, e.g., cone b-metric~\cite{HS11}, b-metric like~\cite{AHS13}, partial b-metric~\cite{S14}. Se also~\cite{MRPK13} and~\cite{Va15} for another modifications and generalizations.


In the present paper we focus on some properties of quasisymmetric mappings between b-metric spaces.

As it is adopted in the theory of quasiconformal mappings, under \emph{embedding} we understand \emph{injective} \emph{continuous} mapping between metric spaces with continuous inverse mapping. In other words, an embedding is a homeo\-morp\-hism on its image.

The definition of $\eta$-quasisymmetry between metric spaces from~\cite{TV80} can be easily generalized to the case of semimetric spaces as follows.

\begin{definition}\label{d1.1}
Let $(X,d)$, $(Y,\rho)$ be semimetric spaces. We shall say that an embedding $f\colon X\to Y$ is $\eta$-\emph{quasisymmetry} if there is a homeomorphism $\eta\colon \mathbb{R}^{+}\to \mathbb{R}^{+}$ so that
\begin{equation}\label{e0}
d(x,a)\leqslant t d(x,b)  \, \text{ implies } \, \rho(f(x),f(a))\leqslant \eta(t)\rho(f(x),f(b))
\end{equation}
for all triples $a,b,x$ of points in $X$ and for every $t \in \mathbb{R}^{+}$.
\end{definition}

Recall that the quantity
$$
\diam A=\sup\{d(x,y)\colon x,y\in A\}.
$$
is the \emph{diameter} of the set $A$ in the semimetric space $(X,d)$.

The following proposition was proved in~\cite{TV80}, see also~\cite[Propositon 10.8]{H01} for the extended proof.

\begin{proposition}\label{p1.2}
Let $X, Y$ be metric spaces and let $f$ be $\eta$-quasi\-sym\-metry. Let $A\subset B\subset X$ with $\diam A>0$, $\diam B<\infty$. Then $\diam f(B)<\infty$ and
\begin{equation}\label{e11}
  \frac{1}{2\eta\left(\frac{\diam B}{\diam A}\right)}\leqslant
  \frac{\diam f(A)}{\diam f(B)}\leqslant
  \eta\left( \frac{2\diam A}{\diam B}\right).
\end{equation}
\end{proposition}
In Theorem~\ref{t2.2}  we prove the analog of this proposition  in the case when $X$ and $Y$ are b-metric spaces.

Metric-preserving functions were in detail studied by many  mathematicians. Recall that a function $f\colon \RR^1\to \RR^1$ preserves metric if the composition $f\circ d$ is a metric on $X$ for any metric space $(X,d)$. Functions preserving metrics and b-metrics were studied, for example, in~\cite{D96,DM13} and in~\cite{KP18, SKP20}, respectively. In this paper we understand the property to preserve b-metricity in another way. In Theorem~\ref{t3.1} we describe $\eta$-quasisymmetries $f$ for which the image $f(X)$ of a b-metric space $X$ with the coefficient $K_1$ is again b-metric space with another coefficient $K_2$.  The analogous question for the case of additive metric spaces is investigated in Theorem~\ref{t4.1}.

\section{Distortion of diameters}

For the proof of the main result of this section we need the following proposition which is a generalization of a corresponding result from~\cite{TV80} to the case of semimetric spaces.

\begin{proposition}\label{p2.1}
Let $(X,d)$ and $(Y,\rho)$ be semimetric spaces.
If $f\colon X\to Y$ is $\eta$-quasisymmetric, then $f^{-1}\colon f(X)\to X$ is $\eta'$-quasisymmetric, where
\begin{equation}\label{e81}
\eta'(t) = 1 / \eta^{-1}(t^{-1})
\end{equation}
 for $t>0$.
\end{proposition}

\begin{proof}
Let $a_1, b_1,x_1 \in Y$ and let $a=f^{-1}(a_1)$, $b=f^{-1}(b_1)$, and $x=f^{-1}(x_1)$. Let us prove the proposition by contradiction. Assume that
$$
\rho(x_1,a_1)\leqslant t\rho(x_1,b_1) \, \text{ but } \,  d(x,a)> \eta'(t)d(x,b).
$$
Then
$d(x,b)< \eta^{-1}(\frac{1}{t})\rho(x,a)$. Using~(\ref{e0}) we get $\rho(x_1,b_1)< \frac{1}{t}\rho(x_1,a_1)$, which contradicts our assumption.
\end{proof}

\begin{theorem}\label{t2.2}
Let $(X,d)$ and $(Y,\rho)$ be $b$-metric spaces with the coefficients $K_1$ and $K_2$, respectively. And let $f\colon X\to Y$ be an $\eta$-quasisymmetry. Then $f$ maps bounded subpaces to bounded subspaces. Moreover, if $A\subseteq B\subseteq X$, $0< \diam A,  \diam B <\infty$,  then $\diam f(B)$ is finite and the double inequality
\begin{equation}\label{e4}
  \frac{1}{2K_2\eta\left(\frac{\diam B}{\diam A}\right)}\leqslant
  \frac{\diam f(A)}{\diam f(B)}\leqslant
  \eta\left( \frac{2K_1\diam A}{\diam B}\right)
\end{equation}
holds.
\end{theorem}

\begin{proof}
Let $(b_n)$ and $(b'_n)$ be sequences such that
$$
\frac{1}{2}\diam B \leqslant d(b_n,b'_n)\to \diam B, \text{ as } n \to \infty.
$$
For every $x\in B$ we have
$$
d(x,b_1)\leqslant \diam B \leqslant 2d(b_1,b'_1)
$$
by~(\ref{e0}) implying
$$
\rho(f(x),f(b_1))\leqslant \eta(2)\rho(f(b_1),f(b'_1)).
$$
In order to see that $\diam f(B)<\infty$ for any $x,y \in B$ consider the inequalities
$$
\rho(f(x),f(y)) \leqslant K_2(\rho(f(x),f(b_1))+ \rho(f(y),f(b_1)))
$$
$$
\leqslant  K_2(\eta(2)\rho(f(b_1),f(b'_1)) + \eta(2)\rho(f(b_1),f(b'_1))) = 2K_2\eta(2)\rho(f(b_1),f(b'_1)).
$$

Let $x, a\in A$. To prove inequality~(\ref{e4}) 
consider the evident inequality
$$
d(a,x)\leqslant \frac{d(x,a)}{d(b_n,a)}d(a,b_n)
$$
which by~(\ref{e0}) implies
\begin{equation}\label{e13}
\rho(f(x),f(a))\leqslant \eta \left( \frac{d(x,a)}{d(b_n,a)} \right)\rho(f(b_n),f(a)).
\end{equation}

Without loss of generality (if needed swapping $b_n$ and $b'_n$) we may assume that
\begin{equation}\label{e56}
d(b'_n,a)\leqslant d(b_n,a).
\end{equation}
Using the triangle inequality
$$
d(b_n, b'_n)\leqslant K_1(d(b_n, a)+d(a, b'_n))
$$
and~(\ref{e56}) we get
$$
d(b_n, b'_n)\leqslant 2 K_1d(b_n, a).
$$
Using the last inequality and the relations $d(x,a)\leqslant \diam A$, $A\subseteq B$, from~(\ref{e13}) we have
$$
\rho(f(x),f(a))\leqslant \eta \left( \frac{2K_1\diam A}{d(b_n, b'_n)} \right)\diam f(B).
$$
Since $ d(b_n,b'_n)\to \diam B$ we have the right inequality in~(\ref{e4}).

By Proposition~\ref{p2.1} $f^{-1}$ is $\eta'$-quasisymmetry.
Since $f(A)\subseteq f(B)\subseteq Y$, $0<\diam f(A),$ $\diam f(B)< \infty$, applying the right inequality in~(\ref{e4}) to $f^{-1}$  we have
$$
\frac{\diam A}{\diam B}\leqslant \eta'\left( \frac{2K_2\diam f(A)}{\diam f(B)}\right).
$$
From~(\ref{e81}) we have
$$
\frac{\diam A}{\diam B}\leqslant \left(\eta^{-1}\left( \frac{\diam f(B)}{2K_2\diam f(A)}\right)\right)^{-1}.
$$
Hence,
$$
\eta^{-1}\left( \frac{\diam f(B)}{2K_2\diam f(A)}\right) \leqslant \frac{\diam B}{\diam A}
$$
and, since $\eta$ is strictly increasing, we have
$$
\frac{\diam f(B)}{2K_2\diam f(A)} \leqslant \eta\left( \frac{\diam B}{\diam A}\right),
$$
which imply the first inequality in~(\ref{e4}). This completes the proof.
\end{proof}

The following corollary gives us a functional inequality imposed on an $\eta$-quasisymmetry between two b-metric spaces.

\begin{corollary}\label{c2.3}
Let $X$ and $Y$ be $b$-metric spaces with the coefficients $K_1$ and $K_2$, respectively, and let $f\colon X\to Y$ be an $\eta$-quasisymmetry. Then for every two subsets $A, B$ of $X$ with  $A\subseteq B$ and $0< \diam A,  \diam B <\infty$ the inequality
\begin{equation*}
2K_2\eta(2K_1t)\eta\left(\frac{1}{t}\right)\geqslant 1,
\end{equation*}
holds, where $t=\diam A / \diam B$.
\end{corollary}

Finally, we get Proposition 10.8 from~\cite{H01} as a following corollary.
\begin{corollary}\label{c2.4}
If $X$ and $Y$ are metric spaces, then the double inequality~(\ref{e4}) holds with $K_1=K_2=1$.
\end{corollary}

\section{The property to preserve b-metricity}

In in the following proposition we describe $\eta$-quasisymmetries $f$ for which the image $f(X)$ of a b-metric space $X$ with the coefficient $K_1$ is again b-metric space with another coefficient $K_2$.

\begin{theorem}\label{t3.1}
Let $(X,d)$ be a b-metric space with the coefficient $K_1$, $(Y,\rho)$ be a semimetric space, and let $f\colon X\to Y$ be a surjective $\eta$-quasisymmetry. If  there exists $K_2\geqslant 1$ such that for every $t_1, t_2\in \RR^+\setminus\{0\}$
\begin{equation}\label{c27}
\left( 1 \leqslant K_1\left(\frac{1}{t_1} + \frac{1}{t_2}\right) \right)
\Rightarrow
\left( 1 \leqslant K_2\left(\frac{1}{\eta(t_1)}+\frac{1}{\eta(t_2)}\right) \right),
\end{equation}
then $\rho$ is a b-metric with the coefficient $K_2$.
\end{theorem}

\begin{proof}
Let $x',y',z' \in Y$ be different points and let $x=f^{-1}(x'), y=f^{-1}(y'), z=f^{-1}(z')$. Hence, $d(x,y)\leqslant K_1(d(x,z)+d(z,y))$ and
\begin{equation*}
1\leqslant K_1\left(\frac{d(x,z)}{d(x,y)}+\frac{d(z,y)}{d(x,y)}\right).
\end{equation*}
Set
$$
\frac{d(x,y)}{d(x,z)}= t_1, \quad \frac{d(x,y)}{d(z,y)}= t_2.
$$
Hence,
\begin{equation}\label{e29}
1\leqslant K_1 \left(\frac{1}{t_1}+\frac{1}{t_2}\right).
\end{equation}
By~(\ref{e0}) we have
$$
\rho(f(x),f(y))\leqslant \eta(t_1) \rho(f(x),f(z)), \quad
\rho(f(x),f(y))\leqslant \eta(t_2) \rho(f(z),f(y))
$$
or equivalently,
\begin{equation}\label{e210}
\rho(x',y')\leqslant \eta(t_1) \rho(x',z'), \quad
\rho(x',y')\leqslant \eta(t_2) \rho(z',y').
\end{equation}

By~(\ref{e29}),~(\ref{c27}) and ~(\ref{e210}) we have
$$
1\leqslant K_2 \left(\frac{1}{\eta(t_1)}+ \frac{1}{\eta(t_2)}\right)
\leqslant K_2 \left(\frac{\rho(x',z')}{\rho(x',y')}+ \frac{\rho(z',y')}{\rho(x',y')}\right)  
$$
$$
\leqslant\frac{1}{\rho(x',y')} K_2(\rho(x',z')+\rho(z',y')).
$$
The inequality $\rho(x',y')\leqslant K_2(\rho(x',z')+\rho(z',y'))$ follows.
\end{proof}

\begin{corollary}\label{c3.2}
Let $(X,d)$ be a b-metric space with the coefficient $K_1$, $(Y,\rho)$ be a semimetric space and let $f\colon X\to Y$ be a surjective $\eta$-quasisymmetry such that:
\begin{itemize}
  \item[(i)] $\eta(u)\eta(v)\leqslant \eta(uv)$;
  \item[(ii)] $\eta$ is subadditive, i.e. $\eta(u+v)\leqslant \eta(u)+\eta(v)$;
  \item[(iii)] $\eta(K_1t)\leqslant K_2 \eta(t)$ for some $K_2\geqslant 1$.
\end{itemize}
Then $Y$ is a b-metric space with the coefficient $K_2$.
\end{corollary}
\begin{proof}
The first inequality in~(\ref{c27}) can be rewritten as
$t_1t_2 \leqslant K_1(t_1 + t_2)$.
Since $\eta$ is strictly increasing we have
$\eta(t_1 t_2) \leqslant \eta(K_1(t_1 + t_2))$.
Using (i) we obtain
$\eta(t_1)\eta(t_2) \leqslant \eta(K_1(t_1 + t_2))$.
Condition (ii) gives
$\eta(t_1)\eta(t_2) \leqslant \eta(K_1t_1) + \eta(K_1 t_2)$.
By (iii) we have $\eta(t_1)\eta(t_2) \leqslant K_2\eta(t_1) + K_2\eta(t_2)$ which is equivalent to the second inequality in~(\ref{c27}).
\end{proof}

The following assertion is well-known, see, e.g., Section 2.12 in~\cite{HLP52}.
\begin{lemma}\label{l3.1}
If $0<\alpha\leqslant 1$, then for $u,v\geqslant 0$ the inequality
$(u+v)^{\alpha}\leqslant u^{\alpha}+v^{\alpha}$
holds.
\end{lemma}

\begin{example}\label{e3.4}
Let in Corollary~\ref{c3.2} $K_1=1$ and let $\eta(t)=t^{\alpha}$, $0<\alpha\leqslant 1$. Then $Y$ is a metric space. Clearly, $\eta$ is strictly increasing, condition (i) holds and (iii) holds with $K_2=1$. Lemma~\ref{l3.1} implies the subadditivity of $\eta$.
\end{example}

\begin{example}\label{e3.5}
Let $(X,d)$ be a $b$-metric space with the coefficient $K_1$, $Y$ be a semimetric space, $f\colon X\to Y$ be $\eta$-quasisymmetric with $\eta(t)=Ct^{\alpha}$, where $C>0$, $0<\alpha \leqslant 1$,  and let the inequality $K_2\geqslant C K_1^{\alpha}$ hold.  Then $Y$ is a $b$-metric space with the coefficient $K_2$.

Indeed, let the first inequality in~(\ref{c27}) hold. Then
$t_1^{\alpha}t_2^{\alpha} \leqslant K_1^{\alpha}(t_1 + t_2)^{\alpha}$. Multiplying both parts of this inequality by $C^2$, using Lemma~\ref{l3.1} and the inequality  $K_2\geqslant C K_1^{\alpha}$ we get
$$
\eta(t_1)\eta(t_2)=C^2t_1^{\alpha}t_2^{\alpha} \leqslant C^2K_1^{\alpha}(t_1 + t_2)^{\alpha} 
$$
$$
\leqslant CK_2(t_1^{\alpha}+t_2^{\alpha})=K_2(\eta(t_1)+\eta(t_2)).
$$
Dividing both parts on $\eta(t_1)\eta(t_2)$ we have the second inequality in~(\ref{c27}) which completes the proof.
\end{example}

The well-known concept of bi-Lipschitz mappings can be easily generalized to the case of semimetric spaces.
\begin{definition}\label{d3.6}
Let $(X,d)$, $(Y,\rho)$ be semimetric spaces. A function $f\colon X \to Y$ is called \emph{bi-Lipschitz} if there exists $L\geqslant 1$ such that the relation
\begin{equation}\label{e31}
\frac{1}{L} d(x,y)\leqslant \rho(f(x),f(y))\leqslant L d(x,y).
\end{equation}
holds for all $x,y \in X$.
\end{definition}

Let $d(x,a)\leqslant td(x,b)$. Using consecutively the right inequality in~(\ref{e31}) the previous inequality and the left inequality in~(\ref{e31}), we get
$$
\rho(f(x),f(a))\leqslant L d(x,a)\leqslant Ltd(x,b)\leqslant L^2t\rho(f(x),f(b)).
$$


This implies the following.
\begin{proposition}\label{p3.6}
For any two semimetric spaces every $L$-bi-Lipschitz embedding is an $\eta$-quasisymmetry with $\eta(t)=L^2t$. 
\end{proposition}

Example~\ref{e3.5} implies the following.
\begin{corollary}\label{c3.7}
Let $X$ be a $b$-metric space with the coefficient $K_1$, $Y$ be a semimetric space, $f\colon X\to Y$ be a $L$-bi-Lipschitz surjective mapping and let the inequality $K_2\geqslant K_1L^2$ holds.  Then $Y$ is a $b$-metric space with the coefficient $K_2$.
\end{corollary}

\section{Quasisymmetric mappings preserving additive metrics}

A metric $d$ on $X$ is \emph{additive}~\cite[p.~7]{DD09} if it satisfies the following strengthened version of the triangle inequality called the four-point inequality:
\begin{equation}\label{e61}
d(x,y) + d(z,u)\leqslant \max\{d(x,z) + d(y,u), d(x,u) + d(y,z)\}
\end{equation}
for all $x$, $y$, $z$, $u\in  X$. Equivalently, among the three sums $d(x,y)+d(z,u)$, $d(x,z)+d(y,u)$, $d(x,u) + d(y,z)$ the two largest sums are equal.
The well-known Buneman's criterion~\cite[Theorem 2]{Bu74} asserts that a finite metric space is additive if and only if it is a tree metric. Note that tree metrics play an important role in phylogenetics~\cite{SS03} and hierarchical clustering~\cite{AC11}.

Recall that an ultrametric is a metric for which the strong triangle inequality $d(x,y)\leqslant \max\{d(x,z),d(z,y)\}$ holds. The class of ultrametric spaces is contained in the class of additive metric spaces.


\begin{theorem}\label{t4.1}
Let $(X,d)$ be an additive metric space, $(Y,\rho)$ be a semimetric space and let $f\colon X\to Y$ be a surjective $\eta$-quasisymmetry. If for every $t_1, t_2, t_3, t_4, t_5\in \RR^+\setminus\{0\}$ the inequality
\begin{equation}\label{e51}
1+\frac{1}{t_1}\frac{1}{t_5}\leqslant \max\left\{\frac{1}{t_1}+\frac{1}{t_2}, \frac{1}{t_3}+\frac{1}{t_4}\right\}
\end{equation}
 implies
\begin{equation}\label{e52}
1+\eta\left(\frac{1}{t_1}\right)\eta\left(\frac{1}{t_5}\right)\leqslant \max\left\{\frac{1}{\eta(t_1)}+\frac{1}{\eta(t_2)}, \frac{1}{\eta(t_3)}+\frac{1}{\eta(t_4)}\right\},
\end{equation}
then $(Y,\rho)$ is also an additive metric space.
\end{theorem}
\begin{proof}
Let $x',y',z',u' \in Y$ be different points and let $x=f^{-1}(x')$, $y=f^{-1}(y')$, $z=f^{-1}(z')$, $u=f^{-1}(u')$.  Dividing both parts of~(\ref{e61}) on $d(x,y)$ and using the equality
$$
\frac{d(z,u)}{d(x,y)}=\frac{d(z,u)}{d(
z,x)}\frac{d(z,x)}{d(x,y)}
$$
we have
\begin{equation}\label{e281}
1 \leqslant \max\left\{\frac{d(x,z)}{d(x,y)} + \frac{d(y,u)}{d(x,y)}, \frac{d(x,u)}{d(x,y)} + \frac{d(y,z)}{d(x,y)}\right\} - \frac{d(z,u)}{d(
z,x)}\frac{d(z,x)}{d(x,y)}.
\end{equation}
Set
$$
\frac{d(x,y)}{d(x,z)}= t_1, \quad \frac{d(x,y)}{d(y,u)}= t_2, \quad \frac{d(x,y)}{d(x,u)}= t_3, \quad \frac{d(x,y)}{d(y,z)}= t_4, \quad \frac{d(z,x)}{d(z,u)}= t_5.
$$
Clearly, ~(\ref{e281}) implies~(\ref{e51}).

By~(\ref{e0}) we have
$$
\rho(f(x),f(y))\leqslant \eta(t_1) \rho(f(x),f(z)), \quad
\rho(f(x),f(y))\leqslant \eta(t_2) \rho(f(y),f(u)),
$$
$$
\rho(f(x),f(y))\leqslant \eta(t_3) \rho(f(x),f(u)), \quad
\rho(f(x),f(y))\leqslant \eta(t_4) \rho(f(y),f(z)),
$$
$$
\rho(f(z),f(u))\leqslant \eta\left(\frac{1}{t_5}\right) \rho(f(z),f(x))
$$
$$
\rho(f(z),f(x))\leqslant \eta\left(\frac{1}{t_1}\right) \rho(f(x),f(y)).
$$
Hence,
$$
\max\left\{\frac{1}{\eta(t_1)} + \frac{1}{\eta(t_2)}, \frac{1}{\eta(t_3)} + \frac{1}{\eta(t_4)}\right\} - \eta\left(\frac{1}{t_1}\right)\eta\left(\frac{1}{t_5}\right)
$$
$$
\leqslant \max\left\{\frac{\rho(x',z')}{\rho(x',y')} + \frac{\rho(y',u')}{\rho(x',y')}, \frac{\rho(x',u')}{\rho(x',y')} + \frac{\rho(y',z')}{\rho(x',y')}\right\} - \frac{\rho(z',u')}{\rho(
z',x')}\frac{\rho(z',x')}{\rho(x',y')}.
$$

Using the last inequality and ~(\ref{e52})  we obtain inequality~(\ref{e61}) for  $x', y', z', u'$  which completes the proof.
\end{proof}
\begin{corollary}\label{c4.2}
If $\eta(t)=t$, then $\eta$-quasisymmetry preserves additive metrics.
\end{corollary}

\bigskip

CONTACT INFORMATION

\medskip
E.~Petrov\\Institute of Applied Mathematics and Mechanics of the NAS of Ukraine, Slovyansk\\eugeniy.petrov@gmail.com

\medskip
R.~Salimov\\ Institute of Mathematics of the NAS of Ukraine, Kiev\\ruslan.salimov1@gmail.com
\end{document}